\documentclass[12pt]{amsproc}
\usepackage{hyperref}

\newtheorem{theorem}{\bf Theorem}[section]
\newtheorem{lemma}[theorem]{\bf Lemma}
\newtheorem{proposition}[theorem]{\bf Proposition}
\newtheorem{corollary}[theorem]{\bf Corollary}

\newtheorem{conjecture}[theorem]{\bf Conjecture}

\theoremstyle{remark}
\newtheorem{claim}{Claim}
\newcommand{\E}{\mathcal E}  

\author[C. Acciarri]{Cristina Acciarri}
\address{Cristina Acciarri:  Department of Mathematics, University of Brasilia,
Brasilia-DF, 70910-900 Brazil}
\email{acciarricristina@yahoo.it}

\author[P. Shumyatsky]{Pavel Shumyatsky}
\address{Pavel Shumyatsky: Department of Mathematics, University of Brasilia, 
70910-900 Bras\'ilia DF, Brazil}
\email{pavel@unb.br}


\thanks{This work was supported by the Conselho Nacional de Desenvolvimento Cient\'{\i}fico e Tecnol\'ogico (CNPq),  and Funda\c c\~ao de Apoio \`a Pesquisa do Distrito Federal (FAPDF), Brazil.}

\title[Cyclic Engel sinks]{On groups in which Engel sinks are cyclic}

\begin{document}
%%%%%%%%%%%%%%%%%%%%%%%%%%%%%%%%%%%%%%%%%%%%%%%%%%%%%%%
%%%%%%%%%%%            Abstract                    %%%%%%%%%%%%%%%%%%%%%%%%%%%%%%
%%%%%%%%%%%%%%%%%%%%%%%%%%%%%%%%%%%%%%%%%%%%%%%%%%%%%%%
\begin{abstract} For an element $g$ of a  group $G$, an Engel sink is a subset  $\E(g)$  such that for every $ x\in G $ all sufficiently long commutators $ [x,g,g,\ldots,g] $ belong to $\E(g)$. We conjecture that if $G$ is a profinite group  in which  every element admits a sink that  is a  procyclic subgroup, then $G$ is procyclic-by-(locally nilpotent). We prove the conjecture in two cases -- when $G$ is a finite group, or a soluble pro-$p$ group.
\end{abstract}

\maketitle
%%%%%%%%%%%%%%%%%%%%%%%%%%%%%%%%%%%%%%%%%%%%%%%%%%%%%%%%
%%%%%%%%%%%%%                       INTRO                   %%%%%%%%%%%%%%%%%%%%%%%%%%
%%%%%%%%%%%%%%%%%%%%%%%%%%%%%%%%%%%%%%%%%%%%%%%%%%%%%%%

\section{Introduction}

\baselineskip18pt

A group $G$ is called an \emph{Engel group} if for every $x,g\in G$ the equation $[x,g,g,\dots , g]=1$ holds, where $g$ is repeated in the commutator sufficiently many times depending on $x$ and $g$. (Throughout the paper, we use the left-normed simple commutator notation
$[a_1,a_2,a_3,\dots ,a_r]=[...[[a_1,a_2],a_3],\dots ,a_r]$.)
Of course, any nilpotent group is an Engel group. For finite groups the converse is also known to be true: a finite Engel group is nilpotent by Zorn's theorem \cite{zorn}.  
Given arbitrary elements $x,g$ in a group $G$, here and in what follows,  for any $n\geq 1$, we will denote  by  $[x,_{\,n} g]$ the commutator of the form $[x,\underset{n}{\underbrace{g,\ldots,g}}]$.

Recently, groups that are `almost Engel' in the sense of restrictions on so-called Engel sinks were given some attention. An Engel sink of an element $g\in  G$ is a set ${\E}(g)$ such that for every $x\in G$ all sufficiently long commutators $[x,g,g,\dots ,g]$ belong to ${\E}(g)$, that is, for every $x\in G$ there is a positive integer $n(x,g)$ such that  \begin{equation*}\label{sink} 
[x,_n g]\in {\E}(g)\qquad \text{for all }n\geq n(x,g).
\end{equation*}
Engel groups are precisely the groups for which we can choose ${\E}(g)=\{ 1\}$ for all $g\in G$. In \cite{khu-shu162} finite, profinite, and compact groups in which every element has a finite Engel sink were considered. It was proved that compact groups with this property are finite-by-(locally nilpotent). Similar result for linear groups was established in \cite{mona18} (see also \cite{wehr} for a shorter proof). Recall that a group $G$ is locally nilpotent if every finitely generated subgroup of $G$ is nilpotent. According to an important theorem, due to Wilson and Zelmanov \cite{wi-ze}, a profinite group is locally nilpotent if and only if it is Engel.

In \cite{glasgo} finite groups in which there is a bound for the ranks of the subgroups generated by  Engel sinks were considered. Recall that the rank of a finite group is the minimum number $r$ such that every subgroup can be generated by $r$ elements. It was shown that if $G$ is a finite group such that for every $g\in G$ the Engel sink ${\E}(g)$ generates a subgroup of rank $r$, then the rank of $\gamma_{\infty}(G)$ is bounded in terms of $r$. Here $\gamma_\infty(G)$ stands for the intersection of all terms of the lower central series of $G$.

The goal of this article is to establish some substantial evidence in favor of the following conjecture.

\begin{conjecture}\label{ups} Let $G$ be a profinite group in which every element admits an Engel sink that generates a procyclic subgroup. Then $G$ is procyclic-by-(locally nilpotent).
\end{conjecture}

First, we consider finite groups in which all elements admit Engel sinks generating cyclic subgroups.

\begin{theorem}\label{mainf} Let $G$ be a finite group in which every element admits an Engel sink generating a cyclic subgroup. Then $\gamma_{\infty}(G)$ is cyclic.
\end{theorem}

Recall that a profinite group is a topological group that is isomorphic to an inverse limit of finite groups. The reader is referred to textbooks  \cite{riza} and  \cite{wi} for background information on profinite groups. In the context of such groups all the usual concepts of groups theory are interpreted topologically. In particular, by a subgroup of a profinite group we always mean a closed subgroup. The next result deals with soluble pro-$p$ groups in which every element admits an Engel sink generating a procyclic subgroup.

\begin{theorem}\label{main} Let $G$ be a soluble pro-p group in which every element admits an Engel sink generating a procyclic subgroup. Then $G$ has a normal procyclic subgroup $K$ such that $G/K$ is locally nilpotent.
\end{theorem}

In the next section we deal with the proof of Theorem \ref{mainf}. The proof of Theorem \ref{main} is given in Section 3. 

\section{Proof of Theorem \ref{mainf}}

We start  with a  collection of  well-known facts about coprime  automorphisms that we will use throughout the article. Given a group $G$ acted on by a group $A$ we write $C_G(A)$ for the subgroup of fixed points of $A$ in $G$ and $[G,A]$ for the subgroup generated by all elements of the form $x^{-1}x^a$, where $x\in G$ and $a\in A$.

\begin{lemma}\label{cc}
Let  $A$ be a group of automorphisms of a finite group $G$ such that $(|G|,|A|)=1$. Then
\begin{enumerate}
\item[(i)] $G=C_{G}(A)[G,A]$;
\item[(ii)] $[G,A,A]=[G,A]$; 
\item[(iii)] $C_{G/N}(A)=C_G(A)N/N$ for any $A$-invariant normal subgroup $N$ of $G$;
\item[(iv)] If $G$ is cyclic of prime-power order, then $A$ is cyclic;
\item[(v)] If $G$ is cyclic of $2$-power order, then $A=1$.
\end{enumerate}
\end{lemma}

The assumption of coprimeness   is unnecessary  in  the following lemma.

\begin{lemma}\label{cyc} Let $G$ be  a cyclic group. The group of automorphisms of $G$ is abelian.
\end{lemma}  

Recall that a normal subgroup $N$ of a finite group $G$ is a normal $p$-complement (for a prime $p$) if $N=O_{p'}(G)$ and $G/N$ is a $p$-group. The well-known theorem of Frobenius states that $G$ possesses a normal $p$-complement if and only if $N_G(H)/C_G(H)$ is a $p$-group for every nontrivial $p$-subgroup $H$ of $G$ (see \cite[Theorem 7.4.5]{go}).

Obviously, in a finite group $G$ every element has the smallest Engel sink, so throughout this section, we use the term Engel sink for the minimal Engel sink, denoted by ${\E}(g)$, of an element $g\in G$. 

\begin{lemma}\label{111} Let $G$ be a finite group in which for each $g\in G$ the Engel sink ${\E}(g)$ generates a cyclic subgroup. Then $G$ has a normal $2$-complement. 
\end{lemma}

\begin{proof} Suppose that this is false. Then $G$ has an element $x$ of odd order and a 2-subgroup $H$ such that $x$ normalizes but not centralizes $H$. Let $E=H\cap \E(x)$. Observe that $x$ normalizes $\langle E\rangle$. In view of Lemma \ref{cc}(v), we deduce that $x$ centralizes $\langle E\rangle$. Therefore for every $h\in H$ we have $[h,x,x,\dots,x]=1$ if $x$ is repeated in the commutator sufficiently many times. In other words, $x$ is Engel in the group $H\langle x\rangle$ and we deduce that $[H,x]=1$. This yields a contradiction.
\end{proof}

In view of the Feit-Thompson Theorem on solubility of groups of odd order \cite{fetho} the following corollary is straightforward.

\begin{corollary}\label{112} Let $G$ be a finite group in which the Engel sink ${\E}(g)$ generates a cyclic subgroup for each $g\in G$. Then $G$ is soluble.
\end{corollary}

Recall that a group $G$ is metanilpotent if it has a normal subgroup $N$ such that both $N$ and $G/N$ are nilpotent. It is easy to see that a finite group $G$ is metanilpotent if and only if $\gamma_{\infty}(G)$ is nilpotent. The next result is well known (see for example  \cite[Lemma 2.4]{AST} for the proof).

\begin{lemma}\label{gama_inf} Let $G$ be a finite metanilpotent group. Assume that $P$ is a Sylow $p$-subgroup of $\gamma_{\infty}(G)$ and $H$ is a Hall $p'$-subgroup of $G$. Then $P=[P,H]$.
\end{lemma}
We will now prove Theorem \ref{mainf} under the additional assumption that $G$ is metanilpotent.

\begin{lemma}\label{h=2} Let $G$ be a finite metanilpotent group in which for each $g\in G$ the Engel sink ${\E}(g)$ generates a cyclic subgroup. Then $\gamma_{\infty}(G)$ is cyclic.
\end{lemma}
\begin{proof} Since $\gamma_{\infty}(G)$ is nilpotent, it is sufficient to show that each Sylow subgroup of $\gamma_{\infty}(G)$ is cyclic. Thus, let $P$ be a Sylow subgroup of $\gamma_{\infty}(G)$ for some prime $p$. In view of Lemma \ref{gama_inf} we have $P=[P,H]$, where $H$ is a Hall $p'$-subgroup of $G$. Without loss of generality we can assume that $G=PH$. Replacing if necessary $P$ by $P/\Phi(P)$ and $H$ by $H/C_H(P)$, we can assume that $P$ is an elementary abelian $p$-group (a vector space over the field with $p$ elements) on which the nilpotent group $H$ acts faithfully by linear transformations.

Taking into account that $H$ is nilpotent, we note that ${\E}(h)=[P,h]$ for each nontrivial $h\in H$. Therefore, if $H=\langle g\rangle$ is cyclic, then $P={\E}(g)$ is cyclic, too. Hence, we assume that $H$ is noncyclic.

Suppose first that $H$ contains a noncyclic abelian subgroup $A$.
 Choose a nontrivial element $a_1\in A$. The cyclic subgroup $[P,a_1]$ is $A$-invariant and, by Lemma \ref{cc}(iv), the quotient $A/C_A([P,a_1])$ is cyclic. In particular $C_A([P,a_1])\neq1$ so we choose a nontrivial element $a_2\in C_A([P,a_1])$. Since $a_2$ centralizes $[P,a_1]$, it follows that $[P,a_1][P,a_2]$ is not cyclic. Moreover, it is clear that $a_1$ centralizes $[P,a_2]$. Hence, $[P,a_1][P,a_2]\leq[P,a_1a_2]$. This shows that ${\E}(a_1a_2)$ is not cyclic, a contradiction. Therefore all abelian subgroups of $H$ are cyclic.

It follows (see for example \cite[Theorem 4.10(ii), p.\ 199]{go}) that $H$ is isomorphic to $Q\times C$, where $Q$ is a generalized quaternion group and $C$ is a cyclic group of odd order. Let $a_{0}$ be the unique involution of $H$. It is clear that $a_0$ is contained in all maximal cyclic subgroups of $H$. Thus we have $[P,h]=[P,a_0]$ for any $h\in H$ and so $[P,H]=[P,a_0]$. Note that $[P,a_0]$ is an $H$-invariant subgroup of order $p$. In view of Lemma \ref{cc}(iv), note that $H$ induces a cyclic group of automorphisms of $[P,a_0]$. We deduce that $a_0$ acts trivially on $[P,a_0]$ and hence on $P$. This is a final contradiction. It shows that $P$ is cyclic, as required.
\end{proof}

Recall that the Fitting height of a finite soluble group $G$ is the minimum number $h=h(G)$ such that $G$ possesses a normal series $1=G_0\leq G_1\leq\dots\leq G_h=G$ all of whose factors are nilpotent. We say that a system of subgroups $P_1,\dots,P_k$ of $G$ is a tower of height $k$ if \begin{itemize}
\item Each subgroup $P_i$ has prime-power order.
\item $P_j$ is normalized by $P_i$ whenever $1\leq i\leq j\leq k$.
\item $P_{i+1}=\gamma_\infty(P_{i+1}P_i)$ for each $i=1,2,\dots,k-1$.
\end{itemize}
Every finite soluble group of Fitting height $h$ possesses a tower of height $h$ (see for example \cite{turull}).

We are now ready to prove the theorem on finite groups.

\begin{proof}[Proof of Theorem \ref{mainf}]
Recall that $G$ is a finite group in which for each $g\in G$ the Engel sink ${\E}(g)$ generates a cyclic subgroup. We need to show that $\gamma_{\infty}(G)$ is cyclic. By Corollary \ref{112}, the group $G$ is soluble. Suppose that the theorem is false and let $G$ be a counter-example of minimal order. Lemma \ref{h=2} shows that $h(G)\geq3$.

Choose three subgroups $P_1,P_2,P_3$ which form a tower of height 3. Since $P_3=\gamma_\infty(P_3P_2)$, because of Lemma \ref{h=2} we conclude that $P_3$ is cyclic. By  Lemma \ref{cyc} the subgroup  $P_2P_1$ induces an abelian group of automorphisms of $P_3$. Since $P_2=\gamma_\infty(P_2P_1)$, we conclude that $P_2$ acts on $P_3$ trivially. In other words, $P_2$ centralizes $P_3$. In view of the equality $P_3=\gamma_\infty(P_3P_2)$ we have a contradiction. This completes the proof.
\end{proof}

\section{Proof of Theorem \ref{main}}
 Our purpose in this section is to prove Theorem \ref{main}. Given an element $g$ of a group $G$, for each $n\geq 1$, we will  denote by $E_n(g)$ the  subgroup  of $G$ generated by all commutators of the form $[x,_{\,n} g]$, with $x$ in $G$. 
  
The next two results, whose proofs can be found in \cite[Lemmas 2.1 and 2.2]{shu2} respectively, state general facts about nilpotent groups and Engel elements.
 
\begin{lemma}\label{ord1} Let $G=H\langle a\rangle$, where $H$ is a normal  nilpotent  subgroup of class $c$ and $a$ is an $n$-Engel element. Then $G$ is nilpotent with class at most $cn$.
\end{lemma}

\begin{lemma}\label{ord2} For any positive integers $c,n$ there exists an integer $f = f(c,n)$ with the following property. Let $G=H\langle a\rangle$, where $H$ is a normal  nilpotent  subgroup of class $c$. Then $\gamma_f(G)\leq E_n(a)$.
\end{lemma}

Here and throughout the article $\gamma_f(G)$ denotes the $f$th term of the lower central series of $G$.

The following  lemma  concerns profinite groups and Engel elements.
 
\begin{lemma}\label{1111} Let $G=M\langle a\rangle$ be a profinite group with an abelian normal subgroup $M$ and an Engel element $a$. Then $G$ is nilpotent.
\end{lemma}

\begin{proof}
For any nonnegative integer $i$ set 
$$B_i=\{x\in M \mid [x,_{\,i} a]=1\}.$$
Each set $B_i$ is closed, and $\bigcup_{i\geq 0} B_i=M$. By Baire's Category Theorem \cite[p.\ 200]{kelly} at least one of these sets has non-empty interior. Therefore there exist an integer $n$, an element $b$ in $M$ and an open normal subgroup $N$ contained in $M$ such that $[y,_{\,n} a]=1$ for any $y\in bN$. From this we deduce that $[x,_{\,n} a]=1$ for any $x$ in $N$.  Since $N$ is open in $M$, there exists a positive integer $k$ such that $[z,_{\,k} a]\in N$ for any $z\in M$. Thus $[M, _{\,n+k} a]=1$ and the result follows. 
\end{proof}

Note that, for  an element $g$ of a group $G$, once a sink $\E(g)$ is chosen, the  subgroup $\langle \E(g)\rangle$ generated by $\E(g)$ is also a sink for $g$. In the remaining part of this article it will be convenient to use the term ``sink $\E(g)$ of $g$'' meaning
a subgroup containing all sufficiently long commutators $[x, _ig]$ with $x\in G$. 

\begin{lemma}\label{1112} Let $G$ be a metabelian profinite group and  let $a$ be an element of $G$. Then, for any choice of a sink  $\E(a)$,  there exists an integer $n$ such that $E_n(a)\leq \E(a)$. 
\end{lemma}
\begin{proof} If $\E(a)$ is finite, then $a$ is  Engel  in $G$. Set $K=G'\langle a\rangle$. By Lemma \ref{1111} the subgroup $K$ is nilpotent and $[G',_{\,n-1} a]=1$, for some integer $n$. Therefore $[G,_{\,n} a]=1$ and so $E_n(a)\leq \E(a)$.

Assume that $\E(a)$ is infinite. Let $E_1$ be the subgroup generated by all commutators $[x,a,\ldots,a]\in \E(a)$ such that $x\in G'$. Note that $E_1\leq\E(a)$ and $E_1$ is  a normal subgroup of $G$. Moreover $a$ is Engel in $G/E_1$. In view of  Lemma \ref{1111} the subgroup $G'\langle a\rangle$ is nilpotent modulo $E_1$ and the result follows.  
\end{proof}

\begin{lemma}\label{1113} Let $G$ be a metabelian profinite group and $a\in G$.  For each $n\geq1$ the subgroup $E_n(a)$ is normal in $G$.
\end{lemma}
\begin{proof} For any $i\geq 1$, any  $g\in G'$ and $y\in G$ we have $$[g,_{\,i} a]^y=[g^y,_{\,i} a]\,\, \text{and}\,\, [g^{-1},_{\,i} a]=[g,_{\,i} a]^{-1}.$$  Moreover, for any $x,y\in G$, the  equality $[x,a]^y=[xy,a][y,a]^{-1}$ holds.  

We only need to prove the lemma with $n\geq2$ since for $n=1$ the result is well known even without the assumption that $G$ is metabelian. For arbitrary elements $x,y\in G$, by using the standard commutator laws, write $$[x,_{\,n} a]^y=[[x,a]^y,_{\,n-1} a]=[[xy,a][y,a]^{-1},_{\,n-1} a]=[xy,_{\,n} a][y,_{\,n} a]^{-1}.$$  The formula above shows that $[x,_{\,n} a]^y\in E_n(a)$ and the lemma follows.
\end{proof}

We write $C_n$ to denote the cyclic group of order $n$ and $\mathbb{Z}_p$ the additive group of $p$-adic integers. Recall that the group of automorphisms of $\mathbb{Z}_p$ is isomorphic to $\mathbb{Z}_p\oplus C_{p-1}$ if $p\geq3$ and $\mathbb{Z}_2\oplus C_2$ if $p=2$ (see for example \cite[Theorem 4.4.7]{riza}). Note that all nontrivial subgroups of $\mathbb{Z}_p$ have finite index in $\mathbb{Z}_p$ (see for example \cite[Proposition and Corollary 1 at p.\ 23]{robert}).

\begin{lemma}\label{nnilpact}
Let $G$ be a pro-$p$ group and $K$ a normal infinite procyclic subgroup of $G$. If $a\notin C_G(K)$, then  for any $i\geq1$, the subgroup $[K,_{\,i} a]$ has finite index in $K$. In particular, if $G$ is locally nilpotent, then $K$ is central in $G$.
\end{lemma}
\begin{proof} Let $\alpha$ be the automorphism of $K$ induced by the conjugation by the element $a$. Write $R$ for the ring of the $p$-adic integers and regard $K$ as the additive group of $R$.  There exists $b\in R$ such that $x^a=x\cdot b$, for each $x\in K$. Note that the subgroup $[K,_{\,n} a]$ consists  of elements of the form $x\cdot(b-1)^n$, where $x$ ranges over $K$. Moreover the set $\{x\cdot(b-1)^n \mid x\in K\}$ is infinite for each $n\geq 1$, since $R$ has no zero divisors. The lemma follows.
\end{proof}

We now can prove Theorem \ref{main} in the particular case where $G$ is metabelian. The general case will require considerably more efforts.
\begin{proposition}\label{1114} Let $G$ be a metabelian pro-$p$ group such that $\E(g)$ can be chosen procyclic for each $g$ in $G$. Then $G$ has a normal procyclic subgroup $K$ such that $G/K$ is locally nilpotent.
\end{proposition}
\begin{proof} If $G$ is Engel, then it is locally nilpotent and there is nothing to prove. Assume that $G$ is not Engel and let $X$ be the set of all non-Engel elements in $G$. In view of  Lemmas \ref{1112} and \ref{1113} we can assume that all $\E(g)$ are chosen procyclic and normal in $G$. Indeed, by Lemmas \ref{1112} and \ref{1113}, for each $g$ in $G$ there exists an  integer $n$ such that $E_n(g)$ is a normal subgroup  in $\E(g)$, so we can take such $E_n(g)$ as the sink $\E(g)$ of $g$. In view of Lemma \ref{nnilpact} each subgroup $[\E(x),_{\,i}x]$ has finite index in $\E(x)$ whenever $x\in X$.

Given $a,b\in X$, suppose  first that $\E(a)\cap \E(b)=1$. On the one hand, $a$ acts on $\E(a)$ in such a way that, for any $i\geq 1$, the subgroup $[\E(a),_{\,i} a]$ has finite index in $\E(a)$. On the other hand, $a$ centralizes $\E(b)$, since the intersection of the two sinks is trivial. A similar remarks applies to $b$.  Note that   $ab$  acts on $\E(a)\oplus\E(b)$ in the following way:  it acts as the element $a$ on $\E(a)$ and as $b$ on $\E(b)$.  This implies that, for any $n$, the subgroup $E_n(ab)$ contains a subgroup, which is the direct sum of a finite index subgroup in $\E(a)$ and a finite index subgroup in $\E(b)$, isomorphic to $\Bbb Z_p\oplus\Bbb Z_p$. Thus $\E(ab)$ is not procyclic, a contradiction. 

Hence, $\E(a)\cap \E(b)\not=1$, for any $a,b\in X$. Let $K=\E(a)$ for some $a\in X$. We see that, for any $g\in G$, the image in $G/K$  of the  sink $\E(g)$ is finite. Indeed, if $g$ is Engel in $G$, then  the claim is obvious. Otherwise $g\in X$  and the image of $\E(g)$ in $G/K$ is isomorphic to $\E(g)/(\E(a)\cap\E(g))$ which is finite.   It follows that $G/K$ is Engel, hence locally nilpotent by the Wilson-Zelmanov theorem. The proof is complete.
\end{proof}

Next, we consider another particular case of  Theorem \ref{main}. 
\begin{lemma}\label{3333} Let $G$ be  a soluble pro-$p$ group such that $\E(g)$ can be chosen procyclic for each $g$ in $G$. Assume that $G$  has a normal nilpotent subgroup $M$ and $a\in G$ such that $G=M\langle a\rangle$. Then there exists $n$ such that $E_n(a)$ is procyclic and normal in $G$. Moreover there exists $i$ such that $\gamma_i(G)$ has finite index in $E_n(a)$.
\end{lemma}
\begin{proof} We argue by induction on the nilpotency class of $M$. If $M$ is abelian, then   in view of Lemma \ref{1113} there exists $n$ such that $E_n(a)$ is procyclic and normal in $G$. Since $G/E_n(a)$ is nilpotent, the result holds.  Suppose that $M$ is nonabelian and set $Z=Z(M)$. By induction assume that there is $n$ such that $L=ZE_n(a)$ is normal in $G$ and $L/Z$ is procyclic. Since $L/Z$ is procyclic, the subgroup $L$ is abelian. Now looking at the action of $\langle a\rangle$ on $L$ and using the fact that $L$ is abelian,  Lemma \ref{1112} shows that if $j$ is big enough, then the subgroup $E_{n+j}(a)$ is procyclic. By  Lemma \ref{ord2} there exists $f$ such that $\gamma_f(G)\leq E_{n+j}(a)$. Thus, $\gamma_f(G)$ is a normal procyclic subgroup and so all subgroups of $\gamma_f(G)$ are normal in $G$. In particular, $E_f(a)$ is normal and procyclic. Since by Lemma \ref{ord1} the factor-group $G/E_f(a)$ is nilpotent, there exists $i$ such that $\gamma_i(G)\leq E_f(a)$. This completes the proof.
\end{proof}

In the sequel we will use, without mentioning explicitly, the following fact: let $H$ be a subgroup of a profinite group $G$ and  let $x$ be an element of $G$ such that $H^x\leq H$. Then $H^x=H$. This is   because  if $H^x<H$, then the inequality would also hold  in some finite image of $G$, which yields a contradiction.

The next result is a key observation that will be applied many times throughout the proof of the main result.
\begin{lemma}\label{lemma_subgr} Let $G$ be a profinite group and $K$ a procyclic pro-$p$ subgroup of $G$ such that $K\cap K^x\neq1$ for each $x\in G$. Then $K$ contains a nontrivial subgroup $L$ (of finite index) which is normal in $G$.
\end{lemma}

\begin{proof} If $K$ is finite, the result is obvious, so we assume that $K$ is infinite. Recall  that $K\cap K^x$ has finite index in $K$ for each $x\in G$. For each $i$ set
$$S_i=\{x\in G\mid  K\cap K^x \text{ has index at most } p^i \text{ in } K\}.$$ The sets $S_i$ are closed. By Baire Category Theorem at least one of these sets has non-empty interior. Therefore there is an open normal subgroup $N$, an element $d\in G$, and a fixed $p$-power $p^i$ such that $K\cap K^x$ has index at most $p^i$ in $K$ for every $x\in dN$. Let $K_0=K^{p^i}$ be the subgroup of index $p^i$ in $K$. We see that $N$ normalizes $K_0$. Since $N$ is open, it follows that $K_0$ has only finitely many conjugates in $G$. Let $L$ be their intersection. Obviously, $L$ is normal in $G$. Since $K_0\cap K_0^x$ has finite index in $K_0$ for each $x\in G$, the subgroup $L$ is nontrivial.
\end{proof}

Now we are ready to deal with  the proof of Theorem \ref{main}. We want to establish that if $G$ is a soluble pro-$p$ group such that $\E(g)$ can be chosen procyclic for each $g$ in $G$, then $G $ has a normal procyclic subgroup $K$ such that $G/K$ is locally nilpotent.

\begin{proof}[Proof of Theorem \ref{main}]
The argument  will be by induction on the derived length  of $G$.  Set $H=G'$. By induction,
$H$ has a normal procyclic subgroup $K$ such that $H/K$ is locally nilpotent. 

\begin{claim}\label{1115} $H$ is locally nilpotent.
\end{claim}

If $K$ is finite, the claim holds. So we assume that $K$ is infinite. It is sufficient to show that $H$ has a procyclic subgroup $K_0$, which is normal in $G$, such that $H/K_0$ is locally nilpotent. Indeed, once the existence of such subgroup $K_0$ is established, observe that $K_0\leq Z(H)$ because $G/C_G(K_0)$ embeds into $\mathrm{Aut}(\mathbb{Z}_p)$ which is abelian. Hence $H$ is locally nilpotent. Thus, assume that $K$ is not normal in $G$.

For any $x\in G$ the quotient $H/K^x$ is locally nilpotent. If there exists $x$ such that $K^x\cap K=1$, then  $H$, being isomorphic to a subgroup $H/K\times H/K^x$, must be locally nilpotent, as desired.  Therefore we will assume that $K^x\cap K\not=1$ for any $x\in G$.

In view of Lemma \ref{lemma_subgr} $K$ contains a nontrivial subgroup $L$  which is normal in $G$.  Since  $H/K$ is locally nilpotent and  $L$ has finite index in $K$, it follows that $H/L$ is locally nilpotent too. Moreover, since $L$ is normal in $G$, it follows that $L$ is in the center of $H$ and  so $H$ is locally nilpotent. This establishes Claim \ref{1115}.

\begin{claim}\label{1116}   Assume  that $G$ has a normal nilpotent subgroup $M$ such that $G/M$ is nilpotent and finitely generated. Then $G$ has a normal procyclic subgroup $M_0$ such that $G/M_0$ is nilpotent.
\end{claim}

Indeed, choose $a_1,\ldots, a_s$ in $G$ such that $G=\langle M,a_1,\ldots,a_s\rangle$. We argue by induction on the nilpotency class of $G/M$ and also use induction on $s$.  

Assume first that $G/M$ is abelian.  The case $s=1$ follows from Lemma \ref{3333} so suppose that $s\geq2$. Let $V_j=M\langle a_j\rangle$, for $1\leq j \leq s$. Observe that for each $j$ the subgroup  $V_j$ is normal in $G$ and, in view of Lemma \ref{3333}, there exists $i(j)$ such that $\gamma_{i(j)}(V_j)$ is procyclic. If for any $j$ the subgroup $\gamma_{i(j)}(V_j)$ is finite (or trivial), then each $V_j$ is nilpotent and so $G$ is nilpotent too. Thus we can assume that some $\gamma_{i(j)}(V_j)$ are procyclic infinite.  Moreover, if for some  $j$ and $k$ the subgroups  $\gamma_{i(j)}(V_j)$ and $\gamma_{i(k)}(V_k)$ are infinite and  satisfy $\gamma_{i(j)}(V_j)\cap\gamma_{i(k)}(V_k)=1$, then we get a contradiction. Indeed, set $N=\gamma_{i(j)}(V_j)\oplus\gamma_{i(k)}(V_k)$ and consider the action of $a_ja_k$ on $N$.  Arguing as in the proof of Proposition \ref{1114}, we see that $\E(a_ja_k)$ is not procyclic, since for any $n$ the subgroup $E_n(a_ja_k)$  contains a subgroup isomorphic to  $\Bbb Z_p\oplus\Bbb Z_p$. We therefore assume that all infinite subgroups $\gamma_{i(j)}(V_j)$ intersect pairwise nontrivially. In particular their intersection $V$ is a nontrivial normal procyclic subgroup such that $G/V$ is nilpotent. This concludes the argument in the case where $G/M$ is abelian.  

Next, suppose that $G/M$ has nilpotency class at least two, so in particular $s$ is bigger than one. Let $W_j=\langle a_j\rangle HM$, for $1\leq j \leq s$.  Note that any subgroup $W_j$ modulo $M$ is a finitely generated subgroup, since it is a subgroup of a finitely generated nilpotent group. Furthermore $W_j$ modulo $M$ has nilpotency class smaller than the nilpotency class of $G/M$, since it is generated by the image of $H$ and $a_j$. Thus, by induction, any $W_j$ has a normal procyclic subgroup $B_j$ such that $W_j/B_j$ is nilpotent. So, there exists $l(j)$ such that $\gamma_{l(j)}(W_j)\leq B_j$. As in the previous paragraph, if all $B_j$ are finite (or trivial), then $G$ is nilpotent. If the infinite $B_j$ intersect nontrivially, then the claim follows since their intersection $B$ is a nontrivial normal procyclic subgroup such that $G/B$ is nilpotent. Suppose that for some $i,j$ the subgroups $B_i$ and $B_j$ are infinite and $B_i\cap B_j=1$. Note that Claim \ref{1115} implies that both $B_i$ and $B_j$ are centralized by $H$. Set $N=B_i\oplus B_j$ and look at the action of $a_ia_j$ on $N$. We see that for any $n$ the subgroup $E_n(a_ia_j)$ contains a subgroup isomorphic to $\Bbb Z_p\oplus\Bbb Z_p$. Thus $\E(a_ia_j)$ is not procyclic, a contradiction.  This concludes the proof of Claim \ref{1116}.
\medskip

Let $R$ be the last nontrivial term of the derived series of $G$. By induction on the derived length of $G$  assume that for $G/R$ the theorem  holds. Thus $G$ has a normal subgroup $S$, containing $R$, such that $S/R$ is procyclic and $G/S$ is locally nilpotent. Obviously, we can choose $S$ in such way that $S\leq H$.  Let $a\in S$ such that $S=R\langle a\rangle$. In view of Claim \ref{1115},  $a$ is an Engel element. Thus applying Lemma \ref{1111} we deduce that $S$ is nilpotent.

\begin{claim}\label{1117} Let $a_1,\dots,a_s\in G$ and set $J=R\langle a_1,\dots,a_s\rangle$. Then $J$ has a normal procyclic subgroup $J_0$ such that $J/J_0$ is nilpotent.\end{claim}

If $J/R$ is nilpotent, then the claim follows from Claim \ref{1116}. Assume that $J/R$ is not nilpotent.   Set $J_1=\langle J,a\rangle$, where $a$ is as above. Note that $S\leq J_1$ and $J_1/S$ is nilpotent since $G/S$ is locally nilpotent. Hence, again by Claim \ref{1116} there exists  a normal procyclic subgroup $N_0$ in $J_1$ such that $J_1/N_0$  is nilpotent. In particular $JN_0/N_0$ is nilpotent too, so we can take $J_0=J\cap N_0$.  This  concludes the proof of Claim \ref{1117}.
\medskip

We now embark on the final part of the proof of the theorem.
 Assume that the group $G$ is not locally nilpotent. Choose elements $a_1,\dots,a_s\in G$ such that $T=\langle a_1,\dots,a_s\rangle$ is not nilpotent. Recall that $S$ is a nilpotent normal subgroup of $G$ such that $G/S$ is locally nilpotent. By Claim \ref{1116} the group $ST$ has a normal procyclic subgroup $K_0$ such that $ST/K_0$ is nilpotent. Without loss of generality we assume that there is a positive integer $i_0$ such that $K_0=\gamma_{i_0}(ST)$. Note that $K_0$ here must be infinite since $T$ is not nilpotent. Moreover we can replace $K_0$ by $S\cap K_0$ and simply assume that $K_0\leq S$. Indeed, since $ST/K_0$ and $ST/S$ are both nilpotent, we have $\gamma_{i}(ST)\leq S\cap K_0 $, for some positive integer $i$.

Given any   finite subset $Y$ of $G$, we write $T_Y$ for the subgroup $\langle Y,T\rangle$. By Claim \ref{1116} the group $ST_Y$ has a normal procyclic subgroup $K_Y$ such that $ST_{Y}/K_Y$ is nilpotent. Again there is a positive integer $i_Y$ such that $K_Y=\gamma_{i_Y}(ST_Y)$. Note that  all subgroups $K_Y$ are infinite and have infinite intersection with $K_0$. Indeed, any subgroup $ST_Y$ contains $ST$, the subgroup $ST$ is nilpotent modulo the intersection of $K_Y$ with $K_0$, so if this intersection were  trivial, then  $ST_Y$ would be nilpotent, a contradiction.  As before, since $G/S$ is locally nilpotent, we choose all $K_Y$ inside $S$.

Now choose an arbitrary element $x\in G$ and set 
$$Y(x)=\{{a_1}^x,\dots,{a_s}^x, a_1,\dots,a_s\}.$$ We see that $K_{Y(x)}$ has infinite intersection with each of the subgroups $K_0$ and $K_0^{x}$. Hence $K_0\cap K_0^x$ is nontrivial and  this holds for any choice of $x\in G$. Thus,  by Lemma \ref{lemma_subgr}, $K_0$ contains a nontrivial subgroup $L_0$ which is normal in $G$.

Note that for any choice of a finite subset $Y$ of $G$, the subgroup $L_0$ intersects  $K_Y$ by a finite index subgroup, since $K_0$ intersects $K_Y$ nontrivially and $L_0$ has finite index in $K_0$.  Therefore every subgroup $T_Y$ is nilpotent modulo $L_0$, since $K_Y$ becomes finite modulo $L_0$.  Hence $G$ is locally nilpotent modulo $L_0$ and  this concludes the proof.
\end{proof}


\begin{thebibliography}{99}

\bibitem{AST} C. Acciarri, P.  Shumyatsky,  A. Thillaisundaram, Conciseness of coprime commutators in finite groups, {\it Bull. Aust. Math. Soc.} {\bf 89} (2014), 252--258. 
\verb#doi:10.1017/S0004972713000361#.

\bibitem{fetho} W. Feit and J. Thompson, Solvability of groups of odd order, 
{\it Pacific J. Math.} {\bf 13} (1963), 775--1029.

\bibitem{go} D. Gorenstein, \ {\it Finite Groups},  Chelsea Publishing Company, New York, 1980.
\bibitem{kelly}J.\,L. Kelly, {\it General Topology}, Van Nostrand, Toronto, New York, London,
1955.

\bibitem{khu-shu162} E. I. Khukhro and P. Shumyatsky, Almost Engel compact groups, {\it J. Algebra}, \textbf {500}  (2018),  439--456. \verb#doi:10.1016/j.jalgebra.2017.04.021#.   

\bibitem{glasgo} E.\,I. Khukhro, P.  Shumyatsky, Finite groups with Engel sinks of bounded rank,  {\it  Glasgow Math. J.}, \textbf{60} (2018), 695--701. \verb#doi:10.1017/S0017089517000404#. 
 
 \bibitem{riza} L. Ribes -- P. Zalesskii, \textit{Profinite Groups}, 2nd Edition, Springer Verlag, Berlin -- New York (2010). 
 
 \bibitem{robert} A.\,M. Robert, \ {\it A Course in p-adic Analysis}, Springer, New York, 2000.

\bibitem{mona18} P. Shumyatsky, Almost Engel linear groups, {\it Monatsh Math}, \textbf {186}  (2018), 711--719. \verb#doi:10.1007/s00605-017-1062-x#.

\bibitem{shu2} P. Shumyatsky, Orderable groups with Engel-like conditions, {\it J. Algebra}, \textbf{499} (2018), 311--320. \verb#doi:10.1016/j.jalgebra.2017.12.018#

\bibitem{turull} A. Turull, Fitting height of groups and of fixed points, \textit{J. Algebra} \textbf{86} (1984), 555-566. \verb#doi:10.1016/0021-8693(84)90048-6#. 

\bibitem{wehr} B.A.F. Wehrfritz, Weak Engel conditions on linear groups, Advances in Group Theory and Applications, (2018) (to appear)
\bibitem{wi} J.\,S. Wilson, Profinite Groups, Clarendon Press, Oxford, 1998.


\bibitem{wi-ze} J. S. Wilson and E. I. Zelmanov, Identities for Lie algebras of pro-$p$ groups, \emph{J. Pure Appl. Algebra} \textbf{81}, no.~1 (1992), 103--109. \verb#doi:10.1016/0022-4049(92)90138-6#.


\bibitem{zorn} M. Zorn,  Nilpotency of finite groups, \emph{Bull. Amer. Math. Soc.} {\bf 42} (1936), 485--486.

\end{thebibliography}
 \end{document}